\documentclass{article}

\usepackage{arxiv}

\usepackage[utf8]{inputenc} 
\usepackage[T1]{fontenc}    
\usepackage{hyperref}       
\usepackage{url}            
\usepackage{booktabs}       
\usepackage{amsfonts}       
\usepackage{nicefrac}       
\usepackage{microtype}      
\usepackage{lipsum}		
\usepackage{graphicx}
\usepackage{natbib}
\usepackage{doi}
\usepackage{multicol}

\newcommand{\tr}{\mathrm{Tr}}
\newcommand{\real}{\mathbb{R}}

\renewcommand{\real}{\mathbb{R}}

\newcommand{\E}{\operatorname{\mathbb{E}}}

\newcommand{\sign}{\mathrm{sign}}
\newcommand{\En}{\mathrm{E}}

\DeclareSymbolFont{bbold}{U}{bbold}{m}{n}
\DeclareSymbolFontAlphabet{\mathbbold}{bbold}

\usepackage{amsmath}
\usepackage{amssymb}
\usepackage{amsthm}
\usepackage{xcolor}
\usepackage{tikz}
\usepackage{hyperref}       
\usepackage{subfig}
\usepackage{mathtools, nccmath}
\usepackage{upgreek}

\newtheorem{theorem}{Theorem}
\newtheorem{lemma}[theorem]{Lemma}

\newtheorem{proposition}[theorem]{Proposition}
\newtheorem{definition}[theorem]{Definition}
\newtheorem{assumption}[theorem]{Assumption}

\definecolor{gnblue4}{RGB}{0,108,212} 
\hypersetup{colorlinks=true, linkcolor=gnblue4, breaklinks=true, urlcolor=gnblue4, citecolor=gnblue4}

\title{Contraction and concentration of measures with applications to theoretical neuroscience}

\author{ \href{https://orcid.org/0009-0000-3444-0838}{\includegraphics[scale=0.06]{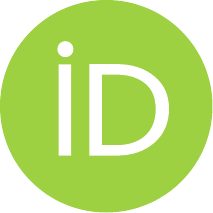}\hspace{1mm}Simone Betteti} \\
	Department of Information Engineering\\
	Università degli Studi di Padova\\
	Padova, 35131, IT \\
	\texttt{bettetisim[at]dei.unipd.it} \\
	\And
	\href{https://orcid.org/0000-0002-4785-2118}{\includegraphics[scale=0.06]{orcid.pdf}\hspace{1mm}Francesco Bullo} \\
	Center for Control, Dynamical Systems and Computation\\
	University of California at Santa Barbara,\\
	Santa Barbara, CA, 93106 IT \\
	\texttt{bullo[at]ucsb.edu}
}

\renewcommand{\headeright}{}
\renewcommand{\undertitle}{}
\renewcommand{\shorttitle}{Contraction and concentration of measures with applications to theoretical neuroscience}

\hypersetup{
pdftitle={Contraction and concentration of measures with applications to theoretical neuroscience},
pdfsubject={math.DS},
pdfauthor={Simone Betteti, Francesco Bullo},
pdfkeywords={Stochastic processes, Wasserstein distance, Contraction theory},
}

\begin{document}
\maketitle

\begin{abstract}
	We investigate the asymptotic behavior of probability measures associated with stochastic dynamical systems featuring either globally contracting or $B_{r}$-contracting drift terms. While classical results often assume constant diffusion and gradient-based drifts, we extend the analysis to spatially inhomogeneous diffusion and non-integrable vector fields. We establish sufficient conditions for the existence and uniqueness of stationary measures under global contraction, showing that convergence is preserved when the contraction rate dominates diffusion inhomogeneity. For systems contracting only outside of a compact set and with constant diffusion, we demonstrate mass concentration near the minima of an associated non-convex potential, like in multistable regimes. The theoretical findings are illustrated through Hopfield networks, highlighting implications for memory retrieval dynamics in noisy environments.
\end{abstract}

\section{Introduction}
The effective functioning of dynamical systems in real world context relies on convergence to desirable states. From cellular reactions to ecosystems, from small circuitry to entire plants, systems evolve dynamically to attain robust and stable configurations, ensuring operational efficiency. Contraction theory~\citep{WL-JJES:98,JWSP-FB:12za,AD-SJ-FB:20o} has recently emerged as a powerful and insightful framework to study the convergence properties of dynamical systems. In essence, contraction theory establishes the asymptotic regularity of a system, whether through convergence to stable equilibria, limit cycles, or trajectories entrainment. Canonically, much of the literature on contraction theory has focused on deterministic system. However, in many practical scenarios, dynamical systems are inherently stochastic, either due to unknown environmental noise or because of the collective behavior of numerous interacting components. Therefore, it is critical to develop formal tools enabling the study of the asymptotic properties of stochastic dynamical systems.

The study of stochastic system asymptotics has a rich and evolving history, and has primarily leveraged the associated probability measure as investigative tool. The canonical approach assumes that the drift field follows the gradient of a potential function~\citep{RJ-DK-FO:97}. More general dynamics consider drift vector fields that are sum of a gradient term and a rotational (divergence-free) term. Under the hypothesis of a constant diffusion term, convergence in Wasserstein distance to a unique stationary probability measure has been established in~\cite{FB-IG-AG:12}. The case of globally contracting but not necessarily integrable drift fields has been investigated under similar diffusion assumptions~\citep{LN-MAP-GS:11}. A first thrust to expand beyond the case of constant diffusivity has been investigated in~\cite{QCP-JJS:13, BJ-JJS:19}, with convergence bounds depending both on the contraction and diffusion constants. Finally, in~\cite{DL-JW:16,PM:23} it has been shown that when drift term is contracting only outside a compact set there exists a unique stationary probability measures and initial distributions converge exponentially fast towards it. Given the pervasiveness of stochastic processes in application, we believe in the need for additional investigation of the associated measures for spatially inhomogeneous diffusivity. Additionally, in our recent study~\citep{SB-GB-FB-SZ:23k} we investigated a stochastic system characterized by a drift term contracting only outside of a compact set and with constant diffusion. Aside from the question on the existence and uniqueness of an associated stationary measure, we were mostly interested in a rigorous understanding of the regions where most of its mass concentrates.

Motivated by these open problems, the paper is structured as follows. Section 2 provides the necessary theoretical background, introducing the concept of c-strong contractivity, unique strong solutions to stochastic differential equations and the associated probability measures. Section 3 investigates stochastic systems with globally contracting drift fields and spatially inhomogeneous diffusion terms, providing sharper bounds on the convergence to a unique stationary measure than in~\cite{BJ-JJS:19}. Section 4 explores mass concentration properties for stochastic systems with drift field contracting outside of a ball centered at the origin. In particular, we prove that the mass concentrates around stable equilibria provided that the basins of attraction and local contraction rates are large enough. In addition, if the non-integrable drift field is associated to a non-convex potential, we prove that the mass tends to concentrates around the potential deepest minima. Finally, section 5 frames the theoretical results in the context of the Hopfield network, a dynamical system describing memory retrieval in the brain known for its multistable nature.

\noindent \textbf{Notation.} We identify with $B_{r}(x)\subset\real^{d}$ the ball of radius $r>0$ and centered at $x\in\real^{d}$. For a smooth manifold $\mathcal{M}\subseteq{\real^{d}}$ we identify with $T_{x}\mathcal{M}$ the tangent space of $\mathcal{M}$ at $x\in\mathcal{M}$, and with $\partial\mathcal{M}$ its boundary. We identify with $C^{k}(\mathcal{X};\mathcal{Y})$ the class of $k$-differentiable functions from $\mathcal{X}$ into $\mathcal{Y}$. Let $f\in C^{1}(\real^{d};\real)$ and denote with $\nabla f(x)\in\real^{d}$ its gradient. We identify the partial derivative of $f$ with respect to $x_{i}$ as $\partial_{x_{i}}f(x)$. Let $g\in C^{1}(\real^{d};\real^{d})$ and denote with $J_{x}g(x)\in\real^{d\times  d}$ its Jacobian, and with $\nabla\cdot g(x)\in\real$ its divergence. We identify the identity matrix of dimension $d\in\mathbb{N}$ as $\mathcal{I}_{d}\in\real^{d\times d}$. We refer to a positive definite matrix $A\in\real^{d\times d}$ as $A\succ 0$. Let $B\in\real^{d\times d}$ and we denote its trace operator as $\tr(B)$. We denote with $(\ ,\ )_{2}:\real^{d}\times\real^{d}\to\real$ the standard Euclidean inner product. We denote the weighted inner product as $(\ ,\ )_{A}:\real^{d}\times\real^{d}\to\real$, for $A\succ 0$. We identify with $\mathcal{P}_{2}(\mathcal{X})$ the space of probability measures over $\mathcal{X}$ having finite second moment. Stochastic processes are identified by boldface upper roman letters, i.e. $\mathbf{X}_{t}$.

\section{Theoretical Background}

Consider the Ito stochastic  differential equation (S.D.E.) 
\begin{equation}\label{eq: SDE-ref}
    \begin{cases}
        d\mathbf{X}_{t} = f(t,\mathbf{X}_{t})\:dt + G(t,\mathbf{X}_{t})\:d\mathbf{B}_{t}\\
        \mathbf{X}(0) = x_{0}\in\real^{d}
    \end{cases}
\end{equation}
where $\mathbf{B}_{t}$ is the standard d-dimensional Brownian motion.
\begin{assumption}[Lipschitzianity and sublinearity]\label{as: LL}
    Let $f:\real\times\real^{d}\rightarrow\real^{d}$ be the drift term, globally Lipschitz with constant $L_{f}>0$ and sublinear with constant $s_{f}>0$.
    \begin{align}
        (f(t,x)-f(t,y),x-y)_{2}&\leq L_{f}\|x-y\|^{2}\qquad\forall x,y\in\real^{d}\\
        \|f(t,x)\|^{2}&\leq s_{f}(1+\|x\|^{2}) 
    \end{align}
    The diffusion term $G:\real\times\real^{d}\rightarrow\real^{d\times d}$ obeys equivalent constraints.
\end{assumption}
Under these standard assumptions~\citep{IK-SS:14} we can guarantee the existence  and uniqueness of a strong solution to eq.~\eqref{eq: SDE-ref}.
\begin{definition}[Infinitesimal generator]
    Let $h:\real\times\real^{d}\rightarrow\real$ be a Lebesgue-measurable function. The infinitesimal generator of $h$ along the process~\eqref{eq: SDE-ref} is defined as
    \begin{align}
        \mathcal{A}h(t,x_{t})=&\lim_{\updelta t\rightarrow0}\frac{\E_{\mathbf{B_{t}}}[h(t+\updelta t, \mathbf{X}_{t+\updelta  t})]-h(t,x_{t})}{\updelta t}\nonumber\\
        =&\partial_{t}h(t,x_{t})+\nabla h(t,x_{t})^{\top}f(x_{t})\ \nonumber\\&+\frac{1}{2}\tr\{G(t,x_{t})\nabla^{2}h(t,x_{t})G(t,x_{t})^{\top}\}\ .
    \end{align}
\end{definition}
The probability measure $\upmu\in\mathcal{P}_{2}(\real_{\geq 0}\times\real^{d})$ is related~\citep{GAP:14} to the S.D.E.~\eqref{eq: SDE-ref} via the Kolmogorov forward equation (or Fokker-Planck equation)
\begin{align}\label{eq: KFE-ref}
        \partial_{t}\upmu(t,x) =& \nabla\cdot\{-f(t,x)\upmu(t,x)+\frac{1}{2}\nabla\cdot[D(t,x)\upmu(t,x)]\}\\
        \upmu(0,x) =& \upmu_{0}(x)\qquad\forall x\in\real^{d}
\end{align}
where $D(t,x)=G(t,x)G(t,x)^{\top}$ and $\upmu_{0}\in \mathcal{P}_{2}(\real^{d})$ is the initial distribution of the states, i.e. the distribution of the initial conditions for eq.~\eqref{eq: SDE-ref}. Additionally, we introduce the notion of distance in $\mathcal{P}_{2}(\real^{d})$ between two probability measures~\citep{CRG-RMS:84}.
\begin{definition}[Wasserstein metric]
    Let $\upmu_{1},\upmu_{2}\in\mathcal{P}_{2}(\real^{d})$ be two probability measures having finite second moment. Define $\Gamma(\upmu_{1},\upmu_{2})$ the set of the joint probability measures $\upgamma$ having $\upmu_{1}$ and $\upmu_{2}$ as marginals. The $p$-Wasserstein metric is defined as
    \begin{align}
        W_{p}(\upmu_{1},\upmu_{2})&=\left(\inf_{\upgamma\in\Gamma(\upmu_{1},\upmu_{2})}\E_{\upgamma}[\|x-y\|^{p}]\right)^{\frac{1}{p}}.
    \end{align}
\end{definition}
For the sake of brevity, from now on we will abbreviate $\Gamma(\upmu_{1},\upmu_{2})$ as simply $\Gamma$, and $\upmu(t,x)=\upmu_{t}(x)$.

\section{Contracting drifts and convergence to stationary measures}
We begin by introducing the definition of global contractivity, and henceforth consider only autonomous (time-independent) drift terms.
\begin{definition}[c-strong contractivity]
    We say that $f:\real^{d}\to\real^{d}$ is c-strongly infinitesimally contracting w.r.t the $2$-norm if it holds
    \begin{equation}
        (f(x)-f(y),x-y)_{2}\leq -c\|x-y\|_{2}\qquad \forall x,y\in\real^{d}
    \end{equation}
    where $c>0$ is the contraction rate.
\end{definition}
We now provide sufficient conditions for the convergence to stationarity of a probability measure~\eqref{eq: KFE-ref} for a $c$-strongly contracting drift term and spatially inhomogeneous diffusion term. 
\begin{theorem}[Contraction of measures]\label{thm: stat}
    Let $\{\mathbf{X}_{t}\}_{t\geq 0}$ be the unique strong solution to eq.~\eqref{eq: SDE-ref} with drift $f$ and diffusion $G$ satisfying Assumption~\ref{as: LL}. Let $f$ be c-strongly contracting w.r.t. the $2$-norm. If $c>L_{G}/2$, then the solution $\upmu_{t}\in\mathcal{P}_{2}(\real^{d})$ to eq.~\eqref{eq: KFE-ref} associated to the process $\{\mathbf{X}_{t}\}_{t\geq 0}$ converges in the 2-Wasserstein metric to a unique stationary probability measure $\upmu^{\star}\in\mathcal{P}_{2}(\real^{d})$.
\end{theorem}
\begin{proof}
    Using a parallel coupling approach~\citep{LN-MAP-GS:11}, we consider two stochastic processes $\{\mathbf{X}_{t}(x_{0},\upomega)\}_{t\geq 0}$ and $\{\mathbf{Z}_{t}(z_{0},\upomega)\}_{t\geq 0}$ both unique strong solutions to eq.~\eqref{eq: SDE-ref} for different initial conditions $x_{0},z_{0}\in\real^{d},\ x_{0}\neq z_{0}$ and the same realization of the noise $\upomega\in\real^{d}$. Taking the standard $2$-norm $h(x)=\|x\|_{2}^{2}$ and evaluating the infinitesimal generator along the difference process $\{\mathbf{X}_{t}-\mathbf{Z}_{t}\}_{t\geq 0}$, we obtain
    \begin{align}
        \mathcal{A}h(&x_{t}-z_{t})=2(f(x_{t})-f(z_{t}),x_{t}-z_{t})_{2}+\tr\{(G(t,x_{t})-G(t,z_{t}))(G(t,x_{t})-G(t,z_{t}))^{\top}\}.
    \end{align}
    Exploiting now the c-strong contractivity of the field $f$ and the Lipschitzianity of $G$
    \begin{equation}
        \mathcal{A}h(x_{t}-z_{t})\leq-\underbrace{(2c-L_{G})}_{>0}h(x_{t}-z_{t}).
    \end{equation}
    Using Dynkin's formula~\citep{BO:13} on the Ito differential $dh(\mathbf{X}_{t}-\mathbf{Z}_{t})$
    \begin{equation}\label{eq: dyn}
        \E_{\mathbf{B}_{t}}[h(\mathbf{X}_{t}-\mathbf{Z}_{t})]-h(x_{0}-z_{0})\leq-\medint\int_{0}^{t}(2c-L_{G})h(\mathbf{X}_{s}-\mathbf{Z}_{s})\:ds.
    \end{equation}
    Applying a specialized version of Gronwall lemma~\citep{QCP:07} we get
    \begin{equation}\label{eq: gron}
        \E_{\mathbf{B}_{t}}[h(\mathbf{X}_{t}-\mathbf{Z}_{t})]\leq h(x_{0}-z_{0})e^{-(2c-L_{G})t}
    \end{equation}
    Let $\upmu_{t}\in\mathcal{P}_{2}(\real^{d})$ be the probability measure associated to the process $\{\mathbf{X}_{t}\}_{t\geq 0}$ and $\upnu_{t}\in\mathcal{P}_{2}(\real^{d})$ be the probability measure associated to the process $\{\mathbf{Z}_{t}\}_{t\geq 0}$. By the positivity of the norm
    \begin{align}\label{eq: norm}
        \E_{\mathbf{B}_{t}}[h(\mathbf{X}_{t}-\mathbf{Z}_{t})]&\geq h(\E_{\mathbf{B}_{t}}[\mathbf{X}_{t}-\mathbf{Z}_{t}])\nonumber\\
        &=h(x_{t}(x_{0})-z_{t}({z_{0}}))
    \end{align}
    where in the last line the processes are independent from the noise realization $\upomega\in\real^{d}$. Let now $\Gamma_{t}$ denote the set of joint probability measures $\upgamma_{t}$ having $\upmu_{t}$ and $\upnu_{t}$ as marginals.
    Taking expectations of eq.~\eqref{eq: gron} w.r.t. the joint measure, and exploiting eq.~\eqref{eq: norm} we get
    \begin{align}
        \E_{\upgamma_{t}}[h(x_{t}-z_{t})]&=\E_{\upgamma_{0}}[h(x_{t}(x_{0})-z_{t}(z_{0}))]\nonumber\\ &\leq \E_{\upgamma_{0}}[h(x_{0}-z_{0})]e^{-(2c-L_{G})t}
    \end{align}
    and miniminzing w.r.t. $\upgamma_{0}$
    \begin{equation}
        \inf_{\upgamma_{o}\in\Gamma_{0}}\E_{\upgamma_{t}}[h(x_{t}-z_{t})]\leq W^{2}_{2}(\upmu_{0},\upnu_{0})e^{-(2c-L_{G})t}.
    \end{equation}
    By the optimality of the infimum w.r.t. the joint measure $\upgamma_{t}\in\Gamma_{t}$ we get
    \begin{equation}
        W^{2}_{2}(\upmu_{t},\upnu_{t})\leq W^{2}_{2}(\upmu_{0},\upnu_{0})e^{-(2c-L_{G})t}
    \end{equation}
\end{proof}
Observe how the presence of spatially inhomogeneous diffusion results in a "discount" of the decay rate to stationarity for the associated measure, which would otherwise equal the contraction rate in the case of homogeneous diffusion. Despite the "discount", the transient towards stationarity retains its exponential character.
Define now the new diffusion term $Q:\real^{d}\to\real^{d\times d}$ with global Lipschitz constant $L_{Q}>0$ and $L = \min\{L_{G},\ L_{Q}\}$.
\begin{assumption}[Noise bound]\label{as: unif}
    The diffusion terms $G$ and $Q$ are uniformly upper bounded in the $2$-norm (Frobenius norm).
    \begin{align}
        \max\{\sup_{t,x} \{\|G(t,x)\|_{F}\},\sup_{t,x}\{\|Q(t,x)\|_{F}\}\}&<+\infty
    \end{align}
\end{assumption}
Then we have the following theorem that bounds the distance in the 2-Wasserstein metric between the probability measures associated to processes with equal drift term and different diffusion term (both processes satisfy the assumptions of Theorem \ref{thm: stat}).
\begin{proposition}[Contraction for different diffusions]\label{prop: diff}
    Let $\{\mathbf{X}_{t}\}_{t\geq 0}$ be the unique strong solution of eq.~\eqref{eq: SDE-ref} with diffusion $G$ and associated measure $\upmu_{t}$. Let $\{\mathbf{Z}_{t}\}_{t\geq 0}$  be the unique strong solution of eq.~\eqref{eq: SDE-ref} with diffusion term $Q$ and associated stationary measure $\upnu_{t}$. Under the hypothesis of Assumption~\ref{as: unif}, if $c>L/2$ then
    \begin{align}
        W^{2}_{2}(\upmu^{\star},\upnu^{\star})&\leq\chi^{2}= \sup_{t,x}\|G(t,x)-Q(t,x)\|_{F}^{2}
    \end{align}
\end{proposition}
\begin{proof}
    Assume without loss of generality that $L=L_{G}$ and consider the difference process $\{\mathbf{X}_{t}-\mathbf{Z}_{t}\}_{t\geq 0}$. Evaluating the infinitesimal generator of $h$ along the difference process we obtain
    \begin{align}
        \mathcal{A}h(x_{t}-z_{t})=&2(f(x_{t})-f(z_{t}),x_{t}-z_{t})_{2}+\|G(t,x_{t})-Q(t,z_{t})\|_{F}^{2}\nonumber\\
        \leq & -2c\ h(x_{t}-z_{t}) + \|G(t,x_{t})-G(t,z_{t})\|_{F}^{2}+\|G(t,z_{t})-Q(t,z_{t})\|_{F}^{2}\nonumber\\
        \leq & -(2c-L)h(x_{t}-z_{t}) + \chi^{2}
    \end{align}
    where in the second passage we have used the strong c-contractivity of $f$ and the triangular inequality, and in the third passage the Lipschitzianity of $G$ and Assumption~\ref{as: unif}. Similarly to Theorem~\ref{thm: stat}, we can apply Dynkin's formula and a specialized version of Gronwall lemma~\citep{QCP:07} to reach
    \begin{align}
        \E_{\mathbf{B}_{t}}[h(\mathbf{X}_{t}-\mathbf{Z}_{t})]&\leq \left(h(x_{0}-z_{0})-\frac{\chi^{2}}{c}\right)^{+}e^{-(2c-L)t}+\chi^{2}\nonumber\\
        &\leq h(x_{0}-z_{0})e^{-(2c-L)t}+\chi^{2}.
    \end{align}
    Taking expectations w.r.t. the joint measures $\upgamma_{t}\in\Gamma_{t}$ and $\upgamma_{0}\in\Gamma_{0}$ and minimizing in the respective sets, we obtain that
    \begin{equation}\label{eq: Was-ineq}
        W^{2}_{2}(\upmu_{t},\upnu_{t})\leq W^{2}_{2}(\upmu_{0},\upnu_{0})e^{-(2c-L)t}+\chi^{2}\quad\forall t\geq 0.
    \end{equation}
    Since $W^{2}_{2}$ is a metric for the space of probability measures, by the triangular inequality we have
    \begin{align}
        W^{2}_{2}(\upmu^{\star},\upnu^{\star})\leq&  W^{2}_{2}(\upmu^{\star},\upmu_{t}) + W^{2}_{2}(\upnu_{t},\upnu^{\star}) + W^{2}_{2}(\upmu_{t},\upnu_{t}).
    \end{align}
    Since it holds $\forall t\geq 0$, passing to the limit and exploiting the convergence of $\upmu_{t}$ and $\upnu_{t}$ we have
    \begin{equation}\label{eq: Was-stat-ineq}
        W^{2}_{2}(\upmu^{\star},\upnu^{\star})\leq \lim_{t\rightarrow +\infty} W^{2}_{2}(\upmu_{t},\upnu_{t}).
    \end{equation}
    Finally, using inequality \eqref{eq: Was-stat-ineq} taking the limit in \eqref{eq: Was-ineq} we get
    \begin{align}
        W^{2}_{2}(\upmu^{\star},\upnu^{\star})\leq \lim_{t\rightarrow +\infty} W^{2}_{2}(\upmu_{0},\upnu_{0})e^{-(2c-L)t}+\chi^{2}
    \end{align}
\end{proof}

\section{$B_{r}$-contracting drifts and concentration of stationary measures}
We are now interested in the analysis of a more general scenario where the drift term has possibly many equilibria (see for example Fig.~\ref{fig:examp}(b)) and is therefore not globally contracting on $\real^{d}$.
\begin{figure}[!h]
    \centering
    \includegraphics[width=.8\linewidth]{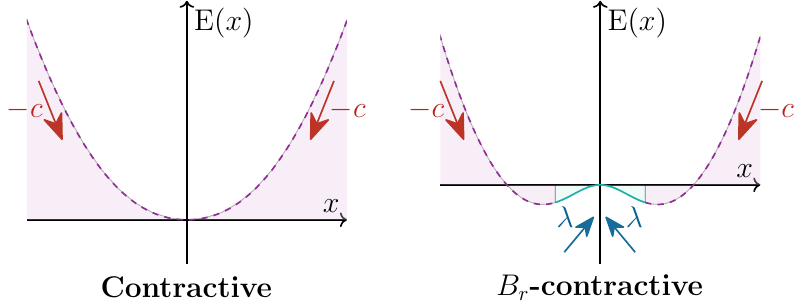}
    \caption{Visual example of potentials $\mathrm{E}(x)$ associated to globally contracting and $B_{r}$-contracting vector fields $f$. (a) Globally contracting vector field associated to a convex potential. (b) $B_{r}$-contractive vector field associated to a mostly convex potential, with a small concave region.}
    \label{fig:examp}
\end{figure}
\begin{definition}[$c$-strong $B_{r}$-contractivity]
    Let $f\in C^{1}(\real^{d},\real^{d})$ and define constants $c,\uplambda>0$ and a radius $r>0$. We say that $f$ is c-strongly $B_{r}$-contracting on $\real^{d}$ if
    \begin{equation}
        (f(x)-f(y),x-y)_{2}\leq
        \begin{cases}
             -c\|x-y\|^{2}\qquad \forall x,y\in B_{r}(0)^{c}\\
             \uplambda\|x-y\|^{2}\qquad\:\: \text{otherwise}
        \end{cases}
    \end{equation}
\end{definition}
The case of $B_{r}$-contractivity is particularly interesting since it encompasses more complex dynamics endowed with multiple equilibrium points. In particular, it allows to relate the concentration of measure around the stable equilibrium points to the properties of their basin of attraction. We will focus on stochastic dynamics characterized by constant noise $G(t,x)\equiv \upomega\mathcal{I}_{d}$, for which it has been proved~\cite{PM:23} that the associated measure converges to stationarity.
\begin{assumption}[Stable equilibria and contraction rates]
	Define $\mathcal{D}=\{x^{\star}\in\real^{D}:\ f(x^{\star})=0, x^{\star}\text{ stable}\}$. We suppose that for each $x^{\star}\in\mathcal{D}$ there exists $r^{\star}>0$ and $c^{\star}>0$ such that
	\begin{equation}
		(f(x)-f(y),x-y)_{2}\leq -c^{\star}\|x-y\|_{2}^{2}\qquad \forall x,y\in B_{r^{\star}}(x^{\star})
	\end{equation}
\end{assumption}
Namely, each stable equilibrium of the $B_{r}$-contracting drift term has a basin of attraction containing a ball of radius $r^{\star}$ where the dynamics are locally contractive in the 2-norm with contraction rate $c^{\star}$. We now prove how the radius $r^{\star}$ and the contraction rate $c^{\star}$ determine if the ball $B_{r^{\star}}(x^{\star})$ attracts or dissipates mass as the measure evolves in time. For notational simplicity, we will henceforth refer to any ball of radius $r>0$ centered at $0\in\real^{d}$ as $B_{r}\equiv B_{r}(0)$, and the integrands will be abbreviated as $f(x)g(x)=(fg)_{|x}$ when necessary.
\begin{proposition}[Mass sinks]\label{prop: van}
    Let $f$ be $c$-strongly $B_{r}$-contracting and $\upmu_{t}\in\mathcal{P}_{2}(\real^{d})$ be the measure associated to the process~\eqref{eq: SDE-ref} via the FPE~\eqref{eq: KFE-ref}. Let $x^{\star}\in\mathcal{D}$ stable equilibrium with $B_{r^{\star}}(x^{\star})$ contained in the basin of attraction. If $c^{\star}\geq\frac{d}{2}\left(\frac{\upomega}{r^{\star}}\right)^{2}$ then
    \begin{equation}
        \medint\int_{B_{r^{\star}}(x^{\star})}\upmu^{\star}(x)\ dx \geq \medint\int_{B_{r^{\star}}(x^{\star})}\upmu_{0}(x)\ dx.
    \end{equation}
\end{proposition}
\begin{proof}
    Assuming sufficient regularity of the measure $\upmu_{t}$ to invert integral and derivatives, we have
    \begin{align}
        &\partial_{t}\medint\int_{B_{r^{\star}}(x^{\star})}\upmu_{t}(x)\ dx\nonumber\\
        =&\medint\int_{B_{r^{\star}}(x^{\star})}\partial_{t}\upmu_{t}(x)\ dx\nonumber\\
        =&\medint\int_{B_{r^{\star}}(x^{\star})}\nabla\cdot\left[-f\upmu_{t}+\frac{\upomega^{2}}{2}\nabla\upmu_{t}\right]_{|x}\ dx\nonumber\\
        =&\medint\int_{\partial B_{r^{\star}}(x^{\star})}\left[-\upmu_{t}(f,\upxi)_{2}+\frac{\upomega^{2}}{2}(\nabla\upmu_{t},\xi)_{2}\right]_{|x}\ d_{x}\upsigma\nonumber\\
        =&\medint\int_{\partial B_{r^{\star}}(x^{\star})}\upmu_{t}\left[-(f,\upxi)_{2}-\frac{\upomega^{2}}{2}\tr(J_{x}\upxi)\right]_{|x}\ d_{x}\upsigma
    \end{align}
    where $d_{x}\upsigma$ is the surface measure, $\upxi(x)\in [T_{x}\partial(B_{\uprho}/B_{R})]^{\perp}$ is the surface outer normal, and in the last passage we used Lemma~\ref{lm: tr}. The outer normal to the shell $\partial B_{r^{\star}}(x^{\star})$ is $\upxi(x)=(x-x^{\star})/r^{\star}$, so we have
    \begin{align}
        &\partial_{t}\medint\int_{B_{r^{\star}}(x^{\star})} \upmu_{t}(x)\ dx\nonumber\\
        =&\frac{1}{r^{\star}}\medint\int_{\partial B_{r^{\star}}(x^{\star})}{\upmu_{t}}_{|x}\left[-(f(x),x-x^{\star})_{2}-\frac{d}{2}\upomega^{2}\right]\ d_{x}\upsigma\nonumber\\
        \geq&\frac{1}{r^{\star}}\medint\int_{\partial B_{r^{\star}}(x^{\star})}{\upmu_{t}}_{|x}\left[c^{\star}\|x-x^{\star}\|_{2}^{2}-\frac{d}{2}\upomega^{2}\right]\ d_{x}\upsigma\nonumber\\
        =&\frac{1}{r^{\star}}\upmu_{t}(\partial B_{r^{\star}}(x^{\star}))\left[c^{\star}{r^{\star}}^{2}-\frac{d}{2}\upomega^{2}\right]\geq 0
    \end{align}
    where from the second to the third passage we have used the fact that $f(x^{\star})=0$ and the contractivity of $f$ in $B_{r^{\star}}(x^{\star})$. Finally, from the third to the last passage we have used the fact that $\|x-x^{\star}\|^{2}_{2}\equiv {r^{\star}}^{2}$ for all $x\in\partial B_{r^{\star}}(x^{\star})$ and denoted $\upmu_{t}(\partial B_{r^{\star}}(x^{\star}))>0$ the probability measure over the surface of the ball. Thus, we have that for all $t>0$
    \begin{equation}
    	\medint\int\displaylimits_{B_{r^{\star}}(x^{\star})}\upmu_{t}(x)\ dx \geq \medint\int\displaylimits_{B_{r^{\star}}(x^{\star})}\upmu_{0}(x)\ dx.
    \end{equation}
    and taking the limit $t\to+\infty$ we conclude.
\end{proof}
The previous result showcases an interesting interplay between the radius $r^{\star}$ of the largest ball $B_{r^{\star}}(x^{\star})$ contained in the basin of attraction of $x^{\star}$ and the contraction rate $c^{\star}$ within the region. Specifically, for equilibria with small basin of attraction, the contraction rate must be significantly large for mass to concentrate within the ball. Conversely, equilibria in regions characterized by weak contraction rates can still attract mass provided that the radius $r^{\star}$ is large enough. We now focus on the simplest generalization of the pure gradient of a potential $\En\in C^{2}(\real^{d};\real)$ such that the $B_{r}$-contracting drift field can be expressed as $f(x)=-P(x)\nabla\En(x)$.
\begin{assumption}[Positivity]\label{as: Pmat}
    The matrix $P\in C^{1}(\real^{d};\real^{d\times d})$ is positive definite $P(x)\succ 0$ for all $x\in\real^{d}$, it is diagonal and $\sign(\partial_{x_{i}}P_{ii}(x))=\sign(x_{i})$.
\end{assumption}
Specifically, we would like to understand whether we can draw useful information on the stationary measure $\upmu^{\star}\in\mathcal{P}_{2}(\real^{d})$ given knowledge on $\En(x)$.
\begin{theorem}[Concentration of mass]\label{thm: mass}
    Let $\{\mathbf{X}_{t}\}_{t\geq 0}$ be the solution of eq.~\eqref{eq: SDE-ref} with $f(x)=-P(x)\nabla\En(x)$ $c$-strongly $B_{r}$-contractive, $P\in C^{1}(\real^{d};\real^{d\times d})$ satisfying Assumption~\ref{as: Pmat}, and $G(t,x)\equiv \upomega\mathcal{I}_{d}$, for $\upomega>0$. Let $x_{a},x_{b}\in\real^{d}$ be minima of the potential $\En(x)$ and choose $r>0$ such that
    \begin{itemize}
        \item[(I)] $\partial B_{r}(x_{a})$ ($\partial B_{r}(x_{b})$) is in the same orthant of $x_{a}$ ($x_{b}$). 
        \item[(II)] $\En(z+x_{a})\leq\En(x_{b})\leq\En(z+x_{b})<0$ for all $z\in B_{r}$.
        \item[(III)] The stationary measure associated to the process $\{\mathbf{X}_{t}\}_{t\geq 0}$ satisfies (weakly)
        $$\nabla\upmu^{\star}(x)=-P(x)\nabla\En(x)\upmu^{\star}(x).$$ 
    \end{itemize} Then the stationary measure satisfies the integral inequality
    \begin{equation}\label{eq: eqmass}
        \medint\int_{B_{r}(x_{a})}\upmu^{\star}(x)\ dx\geq \medint\int_{B_{r}(x_{b})}\upmu^{\star}(x)\ dx.
    \end{equation}
\end{theorem}
\begin{proof}
    The existence and uniqueness of a stationary measure $\upmu^{\star}\in\mathcal{P}_{2}(\real^{d})$ was proved in~\cite{PM:23}, and $\upmu^{\star}$ solves for all $y\in\real^{d}$ the integral equation
    \begin{align}
        0&=\medint\int_{B_{r}(y)}\nabla\cdot\left[P(x)\nabla\En(x)\upmu^{\star}(x)+\frac{\upomega^{2}}{2}\nabla\upmu^{\star}(x)\right]\ dx\nonumber\\
        &=\medint\int_{\partial B_{r}(y)}\upmu^{\star}(x)\left[(\nabla\En(x),\upxi(x))_{P}-\frac{\upomega^{2}}{2}\nabla\cdot\upxi(x)\right]\ d_{x}\upsigma
    \end{align}
    where the unit outer normal is $\upxi(x)=(x-y)/r$, so that $\nabla\cdot\upxi(x)=d/r$. There always exists $\upepsilon>0$ such that $\upomega^{2}d(2r)^{-1}\geq\upepsilon$, and therefore
    \begin{equation}\label{eq: posInt}
        \medint\int_{\partial B_{r}(y)}\upmu^{\star}(x)(\nabla\En(x),\upxi(x))_{P}\ d_{x}\upsigma\geq\upepsilon\medint\int_{\partial B_{r}(y)}\upmu^{\star}(x)\ d_{x}\upsigma
    \end{equation}
    hence the left surface integral is always greater or equal to zero. We now evaluate the difference of the same integral quantities centered on the two minima $x_{a},x_{b}$ of the potential $\En(x)$.
    \begin{align}
        0=&\medint\int_{\partial B_{r}(x_{a})}\upmu^{\star}(x)\left[(\nabla\En(x),\upxi(x))_{P}-\frac{\upomega^{2}d}{2r}\right]\ d_{x}\upsigma-\medint\int_{\partial B_{r}(x_{b})}\upmu^{\star}(x)\left[(\nabla\En(x),\upxi(x))_{P}-\frac{\upomega^{2}d}{2r}\right]\ d_{x}\upsigma\nonumber\\
    \end{align}
    Letting now $c^{r,d}_{\upomega}=2r(\upomega^{2}d)^{-1}$, rearranging the integrals and operating the change of variables $z+x_{a}=x$ (and same for $x_{b}$), we obtain
    \begin{align}\label{eq: traneq}
        &\medint\int_{\partial B_{r}}\upmu^{\star}(z+x_{a})-\upmu^{\star}(z+x_{b})\ d_{z}\upsigma=c^{r,d}_{\upomega}\medint\int_{\partial B_{r}} [\upmu^{\star}(\nabla\En,\upxi)_{P}]_{|z+x_{a}}-[\upmu^{\star}(\nabla\En,\upxi)_{P}]_{|z+x_{b}}\ d_{z}\upsigma
    \end{align}
    Focusing on the integral on the right-hand side and using integration by parts we get
    \begin{align}\label{eq: eneInt}
        &-\medint\int_{\partial B_{r}}[\En\nabla\cdot(P\upxi\upmu^{\star})]_{|z+x_{a}}\ d_{z}\upsigma+\medint\int_{\partial B_{r}}[\En\nabla\cdot(P\upxi\upmu^{\star})]_{|z+x_{b}}\ d_{z}\upsigma 
    \end{align}
    We now focus on the divergence term inside each integrand. Specifically, we have that
    \begin{align}\label{eq: div}
        \nabla\cdot(P\upxi\upmu^{\star})_{|x} &=\Big{[}\sum_{i=1}^{d}\partial_{x_{i}}(P_{ii}\upxi_{i})\Big{]}_{|x}\upmu^{\star}_{|x}+[(P\upxi,\nabla\upmu^{\star})_{2}]_{|x}\nonumber\\
        &=\upmu^{\star}_{|x}\Big{[}\sum_{i=1}^{d}(\partial_{x_{i}}P_{ii})\upxi_{i}+P_{ii}\partial_{x_{i}}\upxi_{i}-(P\upxi,P\nabla\En)\Big{]}_{|x}
    \end{align}
    where from the second to the third passage we have used (III). Since we are evaluating eq.~\eqref{eq: div} along the surface of the ball $\partial B_{r}(x_{a})$ ($\partial B_{r}(x_{b})$ resp.) we have that $[(\partial_{z_{i}}P_{ii})\upxi_{i}]_{z+x_{a}}=r^{-1}[\partial_{z_{i}}P_{ii}(z+x_{a})z_{i}]>0$ (I), $[P_{ii}\partial_{z_{i}}\upxi_{i}]_{|z+x_{a}}=r^{-1}P_{ii}(z+x_{a})>0$. Finally, since the balls are centered at the minima of $\En(x)$, $\nabla\En_{|z+a}\parallel z$ and in particular $\sign(\partial_{z_{i}}\En(z+x_{a}))=-\sign(z_i{})$, from which we conclude that
    \begin{equation}
        \nabla\cdot(P\upxi\upmu^{\star})_{|x}>0\qquad\forall x\in\partial B_{r}(x_{a})\quad (\partial B_{r}(x_{b})\text{ resp.})
    \end{equation}
    The integral on top of eq.~\eqref{eq: eneInt} is positive, as proved by eq.~\eqref{eq: posInt}. Using (II) $\En(z+x_{a})\leq\En(x_{b})\leq \En(z+x_{b})$ in eq.~\eqref{eq: eneInt}, we obtain
    \begin{align}
        &\medint\int_{\partial B_{r}}-[\En\nabla\cdot(P\upxi\upmu^{\star})]_{|z+x_{a}}+[\En\nabla\cdot(P\upxi\upmu^{\star})]_{|z+x_{b}}\ d_{z}\upsigma\nonumber\\
        &\geq\medint\int_{\partial B_{r}} -\En_{|x_{b}}\nabla\cdot[(P\upxi\upmu^{\star})_{|z+x_{a}}-(P\upxi\upmu^{\star})_{|z+x_{b}}]\ d_{z}\upsigma\nonumber\\
        &\geq \En(x_{b})\medint\int_{\partial B_{r}}\nabla\cdot[(P\upxi\upmu^{\star})_{|z+x_{b}}-(P\upxi\upmu^{\star})_{|z+x_{a}}]\ d_{z}\upsigma = 0
    \end{align}
    This means that the right-hand side of eq.~\eqref{eq: traneq} is positive, and therefore
    \begin{equation}
        \medint\int_{\partial B_{r}}\upmu^{\star}(z+x_{a})-\upmu^{\star}(z+x_{b})\ d_{z}\upsigma\geq 0
    \end{equation}
    Finally, given a maximal radius $r_{\max}$ such that the hypothesis (I), (II), and (III) holds, we will have that the conditions also hold for all $r\in[0,r_{\max})$, from which it follows that
    \begin{align}
        &\medint\int_{0}^{r_{\max}}\medint\int_{\partial B_{r}}\upmu^{\star}(z+x_{a})-\upmu^{\star}(z+x_{b})\ d_{z}\upsigma\ dr\nonumber\\
        &=\medint\int_{B_{r_{\max}}(x_{a})}\upmu^{\star}(x)\ dx -\medint\int_{B_{r_{\max}}(x_{b})}\upmu^{\star}(x)\ dx\geq  0. 
    \end{align}
\end{proof}
In plain terms, the theorem states that for systems whose drift is almost the gradient of a potential $\En(x)$, most of the measure mass concentrates around the deepest minima of $\En(x)$. The result well relates to the known case~\cite{AH-KFC-BT-SJM:20} of a stochastic process with $f(x)=-P(x)\nabla\En(x)+\nabla\cdot P(x)$ and $G(t,x)\equiv \sqrt{2}P(x)^{1/2}$, which has associated unique stationary measure $\upmu^{\star}(x)=\mathcal{Z}^{-1}e^{-\En(x)}$ for $\mathcal{Z}$ normalization constant.

\section{Input-driven Hopfield networks and stochastic memory retrieval}
We illustrate the theoretical results by applying them to Hopfield networks, a class of dynamical systems used for memory retrieval that have recently gained renewed interest due to connections with Transformer architectures~\citep{RH-HS:21, BH-YL-BP-RP-HS-DCH-MZ-DK:23}. Specifically, we adopt the framework in~\cite{SB-GB-FB-SZ:23k}, where the interaction between an external input and the connectivity matrix governs whether the dynamics are globally or $B_{r}$-contracting. For expositional and visual clarity, we focus on a two-dimensional system.

\begin{definition}[Input-driven Hopfield model]
    Let $u:\real_{\geq 0}\to\real^{2}_{\geq 0}$ be the constant external input, and define the matrix $M = [1, 1; 1, -1]$. The Input-driven Hopfield dynamics are
    \begin{equation}\label{eq:Hop}\tag{H}
        \begin{cases}
            \dot x = -x+W_{u}\Phi(x)\\
            x(0) = x_{0}\in\real^{2}
        \end{cases}
    \end{equation}
    with activation function $\Phi(x) = (\tanh(\upbeta x_{1}), \tanh(\upbeta x_{2}))$ and $W_{u}=M\text{diag}(u)M^{\top}$.
\end{definition}
Notably, the Hopfield dynamics are associated with an Energy (potential) function of the form
\begin{equation}
    \mathrm{E}(x)=-\frac{1}{2}\Phi(x)^{\top}W_{u}\Phi(x)+x^{\top}\Phi(x)-\sum_{i=1}^{2}\medint\int_{0}^{x_{i}}\Phi_{i}(s)\ ds
\end{equation}
and its dynamics can be rewritten as $\dot x=-J_{x}\Phi(x)^{-1}\nabla\mathrm{E}(x)$, where $J_{x}\Phi(x)^{-1}$ is the inverse of the Jacobian of the activation function. Notice that, for the given activation function, we have $J_{x}\Phi(x)^{-1}_{ii}=[\upbeta(1-\tanh(\upbeta x_{i})^{2})]^{-1}>0$. In addition, $\partial_{x_{i}}(J_{x}\Phi(x)^{-1}_{ii})=2\tanh(\upbeta x_{i})[\upbeta(1-\tanh(\upbeta x_{i})^{2})]^{-1}$, so that $\sign(\partial_{x_{i}}(J_{x}\Phi(x)^{-1}_{ii}))=\sign(x_{i})$. Thus, the conditions of Assumption~\ref{as: Pmat} are satisfied.
\begin{figure}[!tbph]
    \centering
    \subfloat{\includegraphics[width=.5\linewidth]{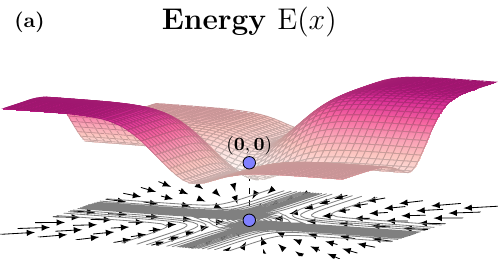}}
    \subfloat{\includegraphics[width=.5\linewidth]{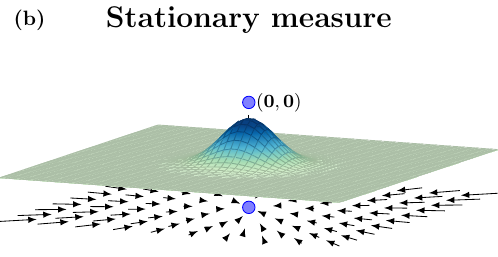}}
    
    \subfloat{\includegraphics[width=.5\linewidth]{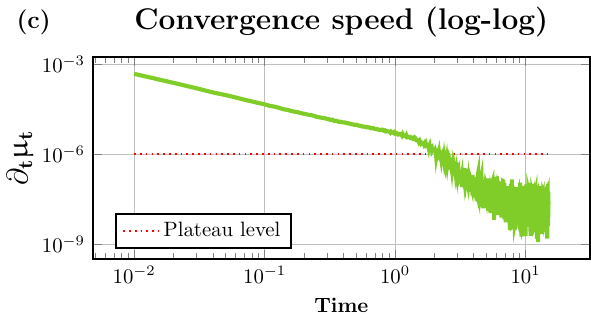}}
    \caption{A stochastic Hopfield model with a globally contracting drift term. (a) Energy function associated to the Hopfield model, with a unique global minimum in the origin. (b) The stationary measure is a gaussian centered at the origin - the globally asymptotically stable equilibrium point of the drift term. The black arrows beneath panels (a-b) are the streamlines associated to the drift $f$. (c) Exponential trend of convergence towards stationarity, conformably with the results of Theorem~\ref{thm: stat}.}
    \label{fig:globalcontr}
\end{figure}

A first result proved in~\cite{SB-GB-FB-SZ:23k} is that if $\max\{u_{1},u_{2}\}<\upbeta^{-1}$, then the system has a unique globally asymptotically stable equilibrium point - the origin. Choosing $\upbeta=2$ and $u_{1}=0.2$, $u_{2}=0.25$, the Hopfield dynamics become globally contracting in the 2-norm (see~\citep[Ch. 2]{FB:24-CTDS} for an overview on Contraction theory) with contraction rate $c\approx0.5$. Choosing a diffusion matrix $G(x)=0.4\cdot\text{diag}(\sin(x_{1}),\cos(x_{2}))$ that is both Lipschitz with $L_{G}=0.32$ and sublinear with $s_{G}=0.16$, we can observe from Fig.~\ref{fig:globalcontr}(b) that the stationary measure is a gaussian centered at the the origin. Furthermore, the logarithmic scale plot in Fig.~\ref{fig:globalcontr}(b) reveals the exponential convergence towards the stationary measure predicted by Theorem~\ref{thm: stat}.

Conversely, when $u_{1}>\upbeta^{-1}$, then $\upgamma_{1}M^{1}$ and $-\upgamma_{1} M^{1}$ are equilibrium points for the dynamics, with $\upgamma_{1}=u_{1}\tanh(\upbeta\cdot\gamma_{1})$, and $M^{1}$ being the first column of $M$. This condition extends to all input entries and all columns of $M$. The global convergence to either of the multiple equilibrium point is guaranteed by the Energy (potential) function $\En(x)$. We additionally notice that $-J_{x}\Phi(x)^{-1}\nabla E(x)$ is non-integrable, and therefore we cannot express the stationary distribution in closed form. Consequently, we decide to verify whether we can estimate where most of the probability mass concentrates. By choosing $3=u_{2}>u_{1}=1>\upbeta^{-1}$ and constant diffusion term $G(x)\equiv .4\cdot\mathcal{I}_{2}$, we observe from Fig.~\ref{fig:semiglobalcontr}(b) that most of the probability mass concentrates around the stable equilibria $\upgamma_{2}M^{2}$ and $-\upgamma_{2}M^{2}$, which are associated to the largest input $u_{2}$. Since it was proven in~\cite{SB-GB-FB-SZ:23k} that if $u_{2}> u_{1}$ then $\mathrm{E}(\pm\upgamma_{2}M^{2})< \mathrm{E}(\pm\upgamma_{1}M^{1})< 0$, the numerical simulation validates the results of Theorem~\ref{thm: mass}. In particular, from Fig.~\ref{fig:semiglobalcontr}(a) it is clear that we can choose a radius $r>0$ such that the hypothesis (I) and (II) of Theorem~\ref{thm: mass} hold. Instead, hypothesis (III) can only be verified numerically. Finally, Fig.~\ref{fig:semiglobalcontr}(b) validates the results in~\cite{PM:23} on the exponential transient towards the stationary distribution.

\begin{figure}[!tpbh]
    \centering
    \subfloat{\includegraphics[width=.5\linewidth]{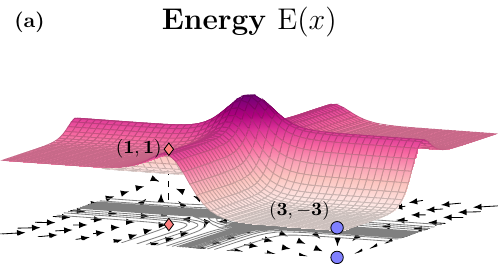}}
    \subfloat{\includegraphics[width=.5\linewidth]{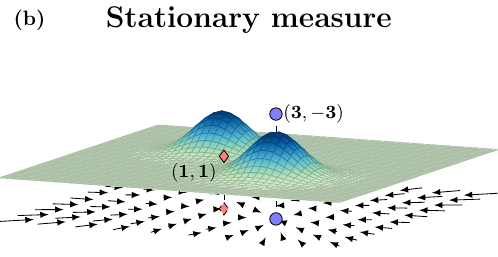}}
    
    \subfloat{\includegraphics[width=.5\linewidth]{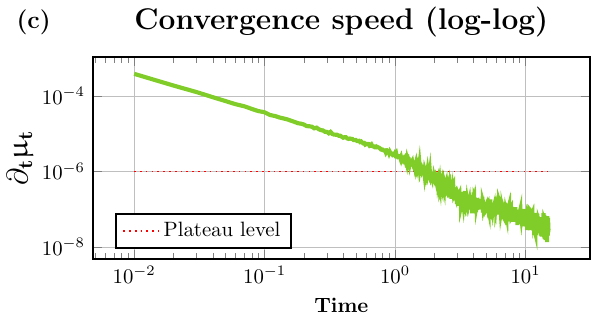}}
    \caption{A stochastic Hopfield model with a $B_{r}$-contracting drift term. (a) Energy associated to the multistable Hopfield model, with antisymmetric minima w.r.t. to the origin. As observable, the minima $(3,-3)$ and $(-3,3)$ associated with the input $u_{2}$ are far deeper than the minima $(1,1)$ and $(-1,-1)$ associated with the input $u_{1}$. (b) The stationary measure concentrates its mass around the stable equilibrium points associated with the deepest Energy $\En(x)$ minima, as predicted by Theorem~\ref{thm: mass}. (c) Exponential trend of convergence towards stationarity, conformably with the results in~\cite{PM:23}.}
    \label{fig:semiglobalcontr}
\end{figure}

Simulations exploit a particle method with $N=2000$ trajectories, Gaussian kernel smoothing (variance $2\cdot \mathcal{I}_2$), and time step $\updelta t=0.01$. The $2$-norm of the difference in the measure over time was used to monitor convergence, with threshold $1\text{e-}6$ indicating stationarity. 

\section{Conclusions}
Despite the intuitive nature of concentration results, asymptotic behavior of measures is a fragile phenomenon that requires careful mathematical manipulation. In this manuscript we have expanded the literature on the convergence of measures associated to globally contracting drift terms and spatially dependent diffusion terms. In particular, we have observed that convergence to a unique stationary measure is guaranteed as long as the contraction rate is twice as big as the Lipschitz constant of the diffusion term. Additionally, we have focused on measures associated to stochastic processes characterized by a $B_{r}$-contracting drift term and constant diffusion term. Theoretical results and numerical simulations reveal how the concentration of mass around the stable equilibria of the drift term depends on the radius of the largest ball centered at the equilibrium and within its basin of attraction as well as the local contraction rate. Moreover, if a non-convex potential $\En(x)$ for the drift term exists, then most of the stationary measure concentrates around its deepest minima. For the sake of completeness, we point out that Proposition~\ref{prop: van} and Theorem~\ref{thm: mass} could be generalized to space of functions with weak derivatives (Sobolev spaces), perhaps at the cost of more burdensome notation. Future works may explore convergence to stationarity of measures associated to $B_{r}$-contracting drift terms and spatially dependent diffusion terms. Finally, numerical investigations may benefit from more extensive experiments using a portfolio of methods for the simulation of partial differential equations as well as dedicated optimal transport  libraries.

\textbf{Acknowledgments:} We would like to thank Prof.\ Katy Craig at UCSB and Lauren E.\ Conger at Caltech for the inspiring discussions held on the topic. The presented work was supported by the grants Next Generation EU C96E22000350007 (S.B.) and AFOSR FA9550-22-1-0059 (F.B.). 

\section*{Appendix}
\begin{lemma}\label{lm: tr}
    Let $\mathcal{M}\subseteq\real^{d}$ be a smooth manifold with boundary, and let $A:\mathcal{M}\to\real^{d\times d}$ and $v:\mathcal{M}\to\real^{d}$ be differentiable. Then
    \begin{equation}
        \medint\int_{\partial\mathcal{M}} (\nabla\cdot A(x),v(x))_{2}\ d_{x}\upsigma = -\medint\int_{\partial\mathcal{M}}\tr(A(x)J_{x}v(x)^{\top})\ d_{x}\upsigma
    \end{equation}
\end{lemma}
\begin{proof}
    The formula for the divergence of a matrix-valued field gives $(\nabla\cdot A(x))_{i}=\sum_{j=1}^{d}\partial_{x_{j}}A_{ij}(x)$ and consequently $(\nabla\cdot A(x),v(x))_{2}=\sum_{i,j=1}^{d}[\partial_{x_{j}}A_{ij}(x)]v_{i}(x)$. Thus, evaluating the integral
    \begin{align}
        &\medint\int_{\partial\mathcal{M}} {(\nabla\cdot A,v)_{2}}_{|x}\ d_{x}\upsigma\ = \medint\int_{\partial\mathcal{M}}\sum_{ij}^{d}[\partial_{x_{j}}A_{ij}]_{|x}{v_{i}}_{|x}\ d_{x}\upsigma\nonumber\\
        =&\sum_{ij}^{d}\medint\int_{\partial\mathcal{M}}\frac{d}{dx_{j}}[A_{ij}v_{i}]_{|x}-[A_{ij}\partial_{x_{j}}v_{i}]_{|x}\ d_{x}\upsigma\nonumber\\ 
        =&\sum_{ij}^{d}\Big{[}\underbrace{\medint\int_{\partial(\partial\mathcal{M})}[A_{ij}v_{i}]_{|x}\ d_{x}\upgamma}_{=0}-\medint\int_{\partial\mathcal{M}}[A_{ij}\partial_{x_{j}}v_{i}]_{|x}\ d_{x}\upsigma\Big{]}\nonumber\\ 
        =&-\medint\int_{\partial\mathcal{M}}\tr(AJ_{x}v^{\top})_{|x}\ d_{x}\upsigma
    \end{align}
    where in the third passage we have used generalized Stokes theorem~\citep{JML:03} on a boundary $\partial\mathcal{M}$ of a smooth manifold $\mathcal{M}$, which has no boundary $\partial(\partial\mathcal{M})$.
\end{proof}

\newpage

\bibliographystyle{unsrtnat}
\bibliography{Main,FB,reference}

\end{document}